\documentclass[a4paper, 10pt]{amsart}

\title{Image of the third Johnson homomorphism}

\author{Quentin Faes}
\author{Ricard Riba}

\address{Institut de Recherche en Mathématique et Physique, Université catholique de
Louvain, Chemin du Cyclotron 2, 1348 Louvain-la-Neuve, Belgium}
\email{quentin.faes@uclouvain.be}

\address{EPSEM, Universitat Politècnica de Catalunya, 08241 Bacelona, Spain}
\email{ricard.riba.garcia@upc.edu}



\date{\today}

\usepackage{amsmath,amsthm,amssymb,graphicx,rotating,
pinlabel,yhmath,tikz,tikz-cd,
listings,xcolor,enumerate,braket,amsfonts,subfiles}

\usepackage[margin=1.5in]{geometry}

\definecolor{specialblue}{RGB}{35,88,180}
\definecolor{lightblue}{RGB}{10,180,255}
\definecolor{verylightblue}{RGB}{213,255,255}

\usepackage{hyperref}
\hypersetup{colorlinks=true,citecolor=specialblue,linkcolor=specialblue,urlcolor=black,filecolor=black}

\usepackage[normalem]{ulem}
\usepackage[all]{xy}

\usetikzlibrary{arrows.meta, decorations.pathmorphing, positioning}
\usetikzlibrary{decorations.pathreplacing}

\allowdisplaybreaks
\numberwithin{equation}{section} 
\SelectTips{cm}{}

\graphicspath{{./Figures/}}

\theoremstyle{definition}
\newtheorem{definition}{Definition}[section]
\newtheorem{remark}[definition]{Remark}

\newtheorem{fact}[definition]{Fact}

\theoremstyle{plain}
\newtheorem{theorem}[definition]{Theorem}
\newtheorem*{theorem*}{Theorem}

\newtheorem{lemma}[definition]{Lemma}
\newtheorem{corollary}[definition]{Corollary}

\newcommand{\calA}{\mathcal{A}}
\newcommand{\calB}{\mathcal{B}}

\newcommand{\calL}{\mathcal{L}}

\newcommand{\calM}{\mathcal{M}}
\newcommand{\calC}{\mathcal{C}}

\newcommand{\Sp}{\operatorname{Sp}}

\newcommand{\id}{\operatorname{id}}
\newcommand{\Ker}{\operatorname{Ker}}
\newcommand{\Tr}{\operatorname{Tr}}

\renewcommand{\Im}{\operatorname{Im}}

\newcommand{\col}{\operatorname{col}}

\makeatletter % To allow usage of special caracters in these macros.
\newcount\countitems
\newcommand{\tree}[1]{%
  \def\mylist{#1}%
  \countitems=0
  \@for\x:=\mylist\do{\advance\countitems by 1\relax}%
  \ifnum\countitems=3
  \def\first{}\def\second{}\def\third{}%
  \@for\x:=#1\do{%
    \ifx\first\empty
      \edef\first{\x}%
    \else\ifx\second\empty
      \edef\second{\x}%
    \else
      \edef\third{\x}%
    \fi\fi
  }%
\begin{tikzpicture}[baseline=8pt ,scale = 0.5]
\draw [thick,color=black] (1,1.5)-- (1,0.75);
\draw [thick,color=black] (1,0.75)-- (1-0.866*0.75,0.75-.37);
\draw [thick,color=black] (1,0.75)-- (1+0.866*0.75,0.75-.37);
\draw [black, fill = black] (1,0.75) circle (.5ex);
\draw[color=black] (1,1.5+0.3) node {\small $\second$};
\draw[color=black] (1-0.866*0.75-0.3,0.75-.37-0.3) node {\small $\first$};
\draw[color=black] (1+0.866*0.75+0.3,0.75-.37-0.3) node {\small $\third$};
\end{tikzpicture}%
\fi
\ifnum\countitems=4
  \def\first{}\def\second{}\def\third{}\def\fourth{}%
  \@for\x:=#1\do{%
    \ifx\first\empty
      \edef\first{\x}%
    \else\ifx\second\empty
      \edef\second{\x}%
    \else\ifx\third\empty
      \edef\third{\x}%
    \else
      \edef\fourth{\x}%
    \fi\fi\fi
  }%
\begin{tikzpicture}[baseline=8pt ,scale = 0.5]
\draw [thick,color=black] (0,1.25)-- (0,0.25);
\draw [thick,color=black] (0,0.75)-- (1,0.75);
\draw [thick,color=black] (1,1.25)-- (1,0.25);
\draw [black, fill = black] (0,0.75) circle (.5ex);
\draw [black, fill = black] (1,0.75) circle (.5ex);
\draw[color=black] (0,0) node {\small $\first$};
\draw[color=black] (0,1.6) node {\small $\second$};
\draw[color=black] (1,1.6) node {\small $\third$};
\draw[color=black] (1,0) node {\small $\fourth$};
\end{tikzpicture}
\fi
 \ifnum\countitems=5
    \def\first{}\def\second{}\def\third{}\def\fourth{}\def\fifth{}%
    \@for\x:=#1\do{%
      \ifx\first\empty \edef\first{\x}%
      \else\ifx\second\empty \edef\second{\x}%
      \else\ifx\third\empty \edef\third{\x}%
      \else\ifx\fourth\empty \edef\fourth{\x}%
      \else \edef\fifth{\x}%
      \fi\fi\fi\fi
    }%
\begin{tikzpicture}[baseline=8pt ,scale = 0.5]
\draw [thick,color=black] (0,1.25)-- (0,0.25);
\draw [thick,color=black] (0,0.75)-- (2,0.75);
\draw [thick,color=black] (2,1.25)-- (2,0.25);
\draw [thick,color=black] (1,0.75)-- (1,1.5);
\draw [black, fill = black] (0,0.75) circle (.5ex);
\draw [black, fill = black] (1,0.75) circle (.5ex);
\draw [black, fill = black] (2,0.75) circle (.5ex);
\draw[color=black] (0,0) node {\small $\first$};
\draw[color=black] (0,1.6) node {\small $\second$};
\draw[color=black] (1,1.75) node {\small $\third$};
\draw[color=black] (2,1.6) node {\small $\fourth$};
\draw[color=black] (2,0) node {\small $\fifth$};
\end{tikzpicture}
 \fi
}

\makeatother

\begin{document}

\begin{abstract}
In this note we show that the image of the third Johnson homomorphism $\tau_3$ coincides with the kernel of Morita trace map for the case of a surface of genus $g$ with one boundary component $\Sigma_{g,1}$, when $g \geq 6$.
Moreover, as a consequence, we get that the image of
$\tau_3$ on the handlebody subgroups coincides with the intersection of kernels of the Morita trace map and the Lagrangian trace maps.
\end{abstract}

\maketitle

\section{Introduction}
Let $\Sigma_{g,1}$ be a surface of genus $g$ with one boundary component. Denote its first homology group $H$, its mapping class group $\mathcal{M}_{g,1}$ and the Torelli subgroup of the latter $\mathcal{I}_{g,1}$. We will most of the time omit the subscripts when the genus is clear. The Johnson filtration $(\mathcal{J}_k)_{k \geq 1}$ of the Torelli group is a famous subject of study and its associated graded embeds in the Lie algebra of symplectic derivations $D(H)$. More precisely the map
\[\oplus \tau_k: \;\bigoplus_{k \geq 1} \mathcal{J}_k/\mathcal{J}_{k+1} \longrightarrow \bigoplus_{k\geq 1} D_k(H)
\]
is a $\Sp(H)$-equivariant Lie ring homomorphism. The image of $\tau_k$ is now more or less well understood after rationalization, but remains a mystery over the integers. For example, torsion in the cokernel is found in even degrees in \cite{Fa25}.

The Enomoto-Satoh trace \cite{ES14}, which generalizes the Morita trace \cite{Mor93}, vanishes on the image of $\tau := \oplus \tau_k$. Actually in degree $3$, the Morita trace will be sufficient. Let us recall its definition (note that we divide by a factor $2$ compared with the original definition).

Let $T(H)$ be the tensor algebra and $S(H)$ the symmetric tensor algebra. The free Lie algebra $\calL(H)$ embeds in the tensor algebra. Recall that $H$ is identified to its dual $H^*$ through the map $x \mapsto \omega(x,-)$, where $\omega$ is the intersection form. The Morita trace
\[
\text{Tr}^M: \; D_3(H) \longrightarrow S^3(H)
\]
is defined as $\frac{1}{2}$ times the composition: 

\[
D_3(H) \longrightarrow H \otimes \mathcal{L}_4(H) \longrightarrow H\otimes T_4(H) \xrightarrow{\sim} H^*\otimes T_4(H)\xrightarrow{\operatorname{cont}}T_3(H) \xrightarrow{\operatorname{proj}}S^3(H)
\] where the map $\operatorname{cont}$ denotes the contraction of $H^*$ with the first tensor in $T_4(H)$.

%%%%%%%%%%%%%%%%%%%%%%%%%%%%%%%%%%%%%%%%%%%%%%%%%%%%%%%%%%%%%%%%
\begin{theorem}
\label{thm:Morita-trace}
For any $g \geq 6$, $\Im(\tau_3) = \Ker(\Tr^M)$, and more precisely, the Morita trace induces an isomorphism $D_3(H)/\Im(\tau_3) \simeq S^3(H) $.
\end{theorem}
%%%%%%%%%%%%%%%%%%%%%%%%%%%%%%%%%%%%%%%%%%%%%%%%%%%%%%%%%%%%%%%%

In this paper, we shall use the tree description of $D(H)$. We refer to \cite[Section 2.1]{FM22} for a precise overview, and to \cite{Lev06} and \cite{CST12} for the historical proofs. A quick reminder is given in Appendix \ref{S:AppendixA}. We will write sum of trees subject to the $\operatorname{AS},\operatorname{IHX}$ and $\operatorname{multilinearity}$ relations to describe elements of $D(H)$.

Fixing a Heegaard splitting of $S^3$, one fixes two handlebodies and we define $\calA$ and $\calB$ to be the subgroups of $\calM_{g,1}$ of mapping classes that extend to the respective handlebodies. This also defines two Lagrangians $A$ and $B$ such that $A \oplus B = H$. In the sequel, we fix a basis $\calC = a_1, a_2, \dots, a_g, b_1, b_2, \dots, b_g$ such that the $a$'s and the $b$'s generate $A$ and $B$ respectively.

Let $W_3(a^{\geq 1}b) \subset D_3(H)$ be the subspace generated by trees that contain at least one label $a$ and $W(ab^4)\subset D_3(H)$ be the subspace generated by trees with one $a$ and four $b$'s in their labels. 
In \cite{Rib25} the second author gave an \textit{antisymmetric Lagrangian trace map}
\[
\Tr^A_\Lambda: W_3(a^{\geq 1}b) \longrightarrow \Lambda^3 B
\]
defined as the composition of maps:
\begin{equation*}
\label{eq:trace-map-A}
\begin{tikzcd}[column sep=4mm]
W_3(a^{\geq 1}b) \ar[r,"\pi_{14}"] & W(ab^4) \ar[r,"\eta_3"] & A\otimes \mathcal{L}_4(B) \ar[r,"i"] & A\otimes T_4(B) \ar[r,"\omega_{1,2}"] & T_3(B) \ar[r] & \Lambda^3 B.
\end{tikzcd}
\end{equation*} where $\pi_{14}$ kills trees with two $a$'s or more, and $\eta_3$ sends a tree to its expansion through the root colored by the letter $a$ (see the reference for a more precise definition). Analogously, we define $\Tr^B$ exchanging $a$'s and $b$'s, and the Lagrangians $A$ and $B$.

Then, by \cite[Prop. 1.1]{Rib25} and Theorem \ref{thm:Morita-trace} we also have an analogous result for the image by $\tau_3$ of the handlebody subgroups $\mathcal{JA}_3:= \mathcal{J}_3 \cap \calA$, $\mathcal{JB}_3:= \mathcal{J}_3 \cap \calB$ and  $\mathcal{JAB}_3:= \mathcal{J}_3 \cap \calA \cap \calB$ of the third term of the Johnson filtration $\mathcal{J}_3$. To be more precise,

%%%%%%%%%%%%%%%%%%%%%%%%%%%%%%%%%%%%%%%%%%%%%%%%%%%%%%%%%%%%%%%%
\begin{theorem}
\label{thm:Lagrangian-trace}
For any $g \geq 6$, we have that
\[
\begin{array}{c}
\tau_3(\mathcal{JA}_3) =\Ker(\Tr^A_\Lambda)\cap \Ker(\Tr^M),\qquad
\tau_3(\mathcal{JB}_3) =\Ker(\Tr^B_\Lambda)\cap \Ker(\Tr^M), \\[1ex]
\tau_3(\mathcal{JAB}_3) =\Ker(\Tr^A_\Lambda)\cap \Ker(\Tr^B_\Lambda)\cap \Ker(\Tr^M).
\end{array}
\]
\end{theorem}
\begin{remark}
Note that there is a slight abuse of notation in the above theorem. The Lagrangian traces $\Tr^A_\Lambda$ and $\Tr^B_\Lambda$ are not defined on the whole space $D_3(H)$ but only on subspaces. Hence to check that a derivation $d$ is in $\tau_3(\mathcal{JA}_3)$, for example, the reader should first check that $d$ belongs to the domain of $\Tr^A_\Lambda$. 
\end{remark}
%%%%%%%%%%%%%%%%%%%%%%%%%%%%%%%%%%%%%%%%%%%%%%%%%%%%%%%%%%%%%%%%

\noindent In the next section we prove Theorem \ref{thm:Morita-trace}.

\section{Proof}

To prove Theorem \ref{thm:Morita-trace}, we shall exhibit an inverse $s$ to the trace map. Let us define it. Let $\mathcal{P}_{\leq n}(X) := \{ A \subseteq X, \textit{ s.t. } \left| A \right|\leq n  \}$. Pick any function $f : \mathcal{P}_{\leq 3}(\{ 1, \dots g \}) \rightarrow \{ 1, \dots g \}$ such that for any $i,j,k \in \{ 1, \dots g \}$, we have $f(\{i,j,k\}) \notin \{i,j,k\}$. It is an elementary fact that such function exists whenever $g \geq 4$.

Using Levine's diagrammatic calculus, recall that (for $a,b,c,d,e \in H$)

\[
\Tr^M \left( \tree{a,b,c,d,e}\right) = \omega(a,e)bcd - \omega (a,d) bce +\omega(b,d)ace - \omega(b,e) acd \in S^3(H).
\]

The number of contractions in an elementary tensor $xyz \in S^3(H)$ is just the number of unordered pairs such that their contractions via $\omega$ is non trivial. The number of contractions in an elementary tree $\tree{a,b,c,d,e}$ is defined in the same way. Let $\col(a_k) = \col(b_k) := k$.

Order the basis $\mathcal{C}$ of $H$, (say by $a_1 \leq b_1 \leq a_2 \dots)$. Define a group homomorphism $s : S^3(H) \rightarrow D_3(H)$ in the following way: 
\begin{itemize}
    \item For any $x\leq y \leq z \in \mathcal{C}$ with no contraction, set
$$s(xyz)= \tree{a_f,x,y,z,b_f} \text{ \hspace{1mm}with $f := f(\{\col(x),\col(y),\col(z)\})$}.$$
\item If there is a single contraction (say $a_k$ with $b_k$) and a third letter $l$, set $$s(a_kb_kl)= \tree{a_f,a_k,b_k,l,b_f} \text{ \hspace{1mm}with $f := f(\{k,\operatorname{col(l)}\})$}.$$
\item If there are two contractions (say with color $k$), set
\[
\begin{aligned}
s(a_ka_kb_k) & = \tree{a_f,a_k,b_k,a_k,b_f}, \\
s(a_kb_kb_k) & = \tree{a_f,b_k,a_k,b_k,b_f},
\end{aligned}
\qquad \text{with } f := f(\lbrace k\rbrace).
\]
\end{itemize}

\noindent Because of the contraction patterns imposed by the map $f$, it is clear that $\Tr^M \circ s = \id_{S^3(H)}$.

%%%%%%%%%%%%%%%%%%%%%%%%%%%%%%%%%%%%%%%%%%%%%%%%%%%%%%%%%%%%%%%%
\begin{corollary}
The Morita trace $\Tr^M$ is surjective when $g \geq 4$.
\end{corollary}
%%%%%%%%%%%%%%%%%%%%%%%%%%%%%%%%%%%%%%%%%%%%%%%%%%%%%%%%%%%%%%%%

It remains to prove the converse, i.e. that $s : S^3(H) \rightarrow D_3(H)/\Im(\tau_3)$ induces a left inverse to the map $\Tr^M$. Luckily, the next lemma tells us exactly how to prove that a derivation is in the image of $\tau_3$.

%%%%%%%%%%%%%%%%%%%%%%%%%%%%%%%%%%%%%%%%%%%%%%%%%%%%%%%%%%%%%%%%
\begin{lemma}
For any $g \geq 3$,
\[
\Im(\tau_3) = \tau_3([\mathcal{J}_2,\mathcal{I}]) = \tau_3(\Gamma_3 \mathcal{I}) = \left[ D_1(H),\left[D_1(H),D_1(H) \right] \right].
\]
\end{lemma}
%%%%%%%%%%%%%%%%%%%%%%%%%%%%%%%%%%%%%%%%%%%%%%%%%%%%%%%%%%%%%%%%

%%%%%%%%%%%%%%%%%%%%%%%%%%%%%%%%%%%%%%%%%%%%%%%%%%%%%%%%%%%%%%%%
\begin{proof}
In \cite{FMS26}, it is proven in Theorem A that $[\mathcal{J}_2,\mathcal{I}] = \Gamma_3 \mathcal{I}$ and in Corollary 4.6 that $\mathcal{J}_3 = \mathcal{J}_4 \cdot \Gamma_3 \mathcal{I}$.
\end{proof}
%%%%%%%%%%%%%%%%%%%%%%%%%%%%%%%%%%%%%%%%%%%%%%%%%%%%%%%%%%%%%%%%

Since elementary trees generate $D_3(H)$, we shall now prove that $T - s(\Tr^M (T))$ belongs to $\Im(\tau_3)$ for any elementary tree $T:= \tree{a,b,c,d,e}$. For this, we distinguish cases according to the number of contractions.
Let $\omega_i\in Sp(H)$ be the element that sends $a_i$ to $b_i$ and $b_i$ to $-a_i$ and leaves all the other basis elements of $H$ fixed. We define $W_g$ to be the subgroup of $\Sp(H)$ generated by $\omega_1,\omega_2, \ldots, \omega_g$, and set $\mathfrak{N}_g=W_g\cdot \mathfrak{S}_g$
In all cases we take $g\geq 6$. To simplify the computations, we first prove that we can use the action of $\mathfrak{N}_g$. Two leaves are called \emph{distant} if they are not related to the same trivalent vertex. We also set $a_k^* = b_k$ and $b_k^* = a_k$.

%%%%%%%%%%%%%%%%%%%%%%%%%%%%%%%%%%%%%%%%%%%%%%%%%%%%%%%%%%%%%%%%
\begin{lemma}
\label{lema:trees-without-contract}
%%%%%%%%%%%%%%%%%%%%%%%%%%%%%%%%%%%%%%%%%%%%%%%%%%%%%%%%%%%%%%%%
The trees of $D_3(H)$ without contractions between distant leaves are in $\Im(\tau_3)$. That is, the following trees are in $\Im(\tau_3)$.
\[
i)\; \tree{a,b,c,d,e}, \qquad ii)\; \tree{a_1,b_1,c,d,e}, \qquad iii)\; \tree{a_1,b_1,c,a_2,b_2}, \]  when there is no contraction except the explicit ones.
\end{lemma}
%%%%%%%%%%%%%%%%%%%%%%%%%%%%%%%%%%%%%%%%%%%%%%%%%%%%%%%%%%%%%%%%

%%%%%%%%%%%%%%%%%%%%%%%%%%%%%%%%%%%%%%%%%%%%%%%%%%%%%%%%%%%%%%%%
\begin{proof}
Let $f\in \mathcal{C}$ such that $\col(f) \notin \lbrace \col(a),\col(b),\col(c),\col(d),\col(e) \rbrace $, we have that
\[
\tree{a,b,c,d,e}=\Bigg[ \tree{a,b,f},\Bigg[ \tree{f^*,c,f},\tree{f^*,d,e}\Bigg]\Bigg].
\]
Hence, the trees of type $i)$ belong to $\Im(\tau_3)$. The same kind of computation shows that the trees of types $ii)$ and $iii)$ also belong to $\Im(\tau_3)$.
\end{proof}
%%%%%%%%%%%%%%%%%%%%%%%%%%%%%%%%%%%%%%%%%%%%%%%%%%%%%%%%%%%%%%%%

%%%%%%%%%%%%%%%%%%%%%%%%%%%%%%%%%%%%%%%%%%%%%%%%%%%%%%%%%%%%%%%%
\begin{lemma}
\label{lema:trees-in-image}
%%%%%%%%%%%%%%%%%%%%%%%%%%%%%%%%%%%%%%%%%%%%%%%%%%%%%%%%%%%%%%%%
The following trees belong to $\Im(\tau_3)$
\[
i)\; \tree{x,a_1,b_1,y,z}, \qquad ii)\; \tree{x,a_1,b_1,a_2,b_2}, \qquad iii)\; \tree{x,a_1,b_1,a_1,y}, \] where the only contractions are the explicit ones.
\end{lemma}
%%%%%%%%%%%%%%%%%%%%%%%%%%%%%%%%%%%%%%%%%%%%%%%%%%%%%%%%%%%%%%%%

%%%%%%%%%%%%%%%%%%%%%%%%%%%%%%%%%%%%%%%%%%%%%%%%%%%%%%%%%%%%%%%%
\begin{proof}
For trees of type $i)$ and $ii)$, we have that, for $k \notin \{ 1, 2, \col(x), \col(y), \col(z)\}$,
\[ \;\tree{x,a_1,b_1,y,z}=\Bigg[\tree{x,a_1,b_1,a_k}+\frac{1}{2}\tree{x,a_k,x,a_k},\tree{b_k,y,z}\Bigg]
-\tree{x,a_k,x,y,z},
\]

\[ \;\tree{x,a_1,b_1,a_2,b_2}=\Bigg[\tree{x,a_1,b_1,a_k}+\frac{1}{2}\tree{x,a_k,x,a_k},\tree{b_k,a_2,b_2}\Bigg]
-\tree{x,a_k,x,a_2,b_2},
\]
where the trees of degree 2 in the brackets belong to $\Im(\tau_2)$ by \cite{Fa23} and the second tree in both equalities belong to $\Im(\tau_3)$ by Lemma \ref{lema:trees-without-contract}.

\noindent Finally, for trees of type $iii)$, consider the tree $\tree{x,a_1,b_1,a_2,y}$ which belongs to $\Im(\tau_3)$ by point $i)$. Consider $G\in \Sp(H)$ that sends $a_2$ to $a_2+a_1$ and $b_1$ to $b_1-b_2$.
Then the following element also belongs to $\Im(\tau_3)$.
\[
G\Bigg(\tree{x,a_1,b_1,a_2,y}\Bigg)=
\tree{x,a_1,b_1,a_2,y}
+\tree{x,a_1,b_1,a_1,y}
-\tree{x,a_1,b_2,a_2,y}
-\tree{x,a_1,b_2,a_1,y}.
\]
Hence, by point $i)$ we get that
\[
\tree{x,a_1,b_1,a_1,y}\in \Im(\tau_3).
\]

\end{proof}
%%%%%%%%%%%%%%%%%%%%%%%%%%%%%%%%%%%%%%%%%%%%%%%%%%%%%%%%%%%%%%%%

%%%%%%%%%%%%%%%%%%%%%%%%%%%%%%%%%%%%%%%%%%%%%%%%%%%%%%%%%%%%%%%%
\begin{lemma}
\label{lema:equivariance}
%%%%%%%%%%%%%%%%%%%%%%%%%%%%%%%%%%%%%%%%%%%%%%%%%%%%%%%%%%%%%%%%
For any $h \in \mathfrak{N}_g$ and $w \in S^3(H)$, $h \cdot s(w) - s(h \cdot w) \in \Im(\tau_3)$. Hence, $T-s(\Tr^M(T)) \in \Im(\tau_3)$ if and only if $h \cdot T-s(\Tr^M(h\cdot T)) \in \Im(\tau_3)$.
\end{lemma}
%%%%%%%%%%%%%%%%%%%%%%%%%%%%%%%%%%%%%%%%%%%%%%%%%%%%%%%%%%%%%%%%

%%%%%%%%%%%%%%%%%%%%%%%%%%%%%%%%%%%%%%%%%%%%%%%%%%%%%%%%%%%%%%%%
\begin{proof}
The action of $\mathfrak{N}_g$ does not affect the contraction patterns since $W_g$ merely exchanges $a_i$'s and $b_i$'s (up to sign), while $\mathfrak{S}_g$ permutes colors. Hence we proceed case by case in terms of the number of contractions of $w$. When there is no contraction, and $x \leq y \leq z$ are the labels of $h\cdot w$, it is enough to prove that for any $f,f'$ such that $\lbrace a_f,b_f,a_{f'},b_{f'}\rbrace$ have no contraction with $\lbrace x,y,z \rbrace$, the elements
\[ i)\;\tree{a_f,x,y,z,b_f}- \tree{a_{f'},x,y,z,b_{f'}} \qquad  ii)\;\tree{a_f,x,y,z,b_f}- \tree{a_{f'},y,x,z,b_{f'}} \qquad iii)\; \tree{a_f,x,y,z,b_f}- \tree{a_{f'},x,z,y,b_{f'}}
\]
belong to $\Im(\tau_3)$.

\textbf{Type i)} Consider the tree $\tree{a_f,x,y,z,b_{f'}}$ which belongs to $\Im(\tau_3)$ by Lemma \ref{lema:trees-without-contract}. Take the action by $G\in \Sp(H)$ that sends $a_f$ to $a_f-a_{f'}$ and $b_{f'}$ to $b_{f'}+b_f$. Then we have that the following element also belongs to $\Im(\tau_3)$.
\[
G\Bigg(\tree{a_f,x,y,z,b_{f'}}\Bigg)= \tree{a_f,x,y,z,b_{f'}}+\tree{a_f,x,y,z,b_{f}}-
\tree{a_{f'},x,y,z,b_{f'}}-\tree{a_{f'},x,y,z,b_{f}},
\]
and hence, again by Lemma \ref{lema:trees-without-contract}, we have that
$\tree{a_f,x,y,z,b_f}-\tree{a_{f'},x,y,z,b_{f'}}$ belongs to $\Im(\tau_3)$.

\textbf{Type ii)} By the AS and IHX relations and point $i)$, modulo $\Im(\tau_3)$, we have that
\[
\tree{a_f,x,y,z,b_f}- \tree{a_{f'},y,x,z,b_{f'}} =
\tree{a_f,x,y,z,b_f}- \tree{a_{f'},x,y,z,b_{f'}}- \tree{x,y,a_{f'},z,b_{f'}} \equiv \tree{z,b_{f'},a_{f'},y,x}. 
\]
This element is in the $\mathfrak{N}_g$-orbit of a tree of type $i)$ in Lemma \ref{lema:trees-in-image} and hence in $\Im(\tau_3)$.

\textbf{Type iii)} By the AS and IHX relations and point $i)$, modulo $\Im(\tau_3)$, we have that
\[
\tree{a_f,x,y,z,b_f}- \tree{a_{f'},x,z,y,b_{f'}} =
\tree{a_f,x,y,z,b_f}- \tree{a_{f'},x,y,z,b_{f'}}- \tree{a_{f'},x,b_{f'},y,z} \equiv \tree{x,a_{f'},b_{f'},y,z}
\]
which is in the $\mathfrak{N}_g$-orbit of a tree of type $i)$ in Lemma \ref{lema:trees-in-image}  and hence in $\Im(\tau_3)$.

When there is a single contraction, $h\cdot w=a_kb_kl$ with $a_k<b_k<l$ or $l<a_k<b_k$. Then it is enough to prove that
\[
iv)\;\tree{a_f,a_k,b_k,l,b_f}- \tree{a_{f'},a_k,b_k,l,b_{f'}} \qquad v)\;\tree{a_{f},a_k,b_k,l,b_{f}}
+\tree{b_{f'},a_k,b_k,l,a_{f'}}
\qquad vi)\;\tree{a_f,a_k,b_k,l,b_f} - \tree{a_{f'},b_k,a_k,l,b_{f'}}
\]
belong to $\Im(\tau_3)$.

\textbf{Type iv)} Consider the tree $\tree{a_{f'},a_k,b_k,l,b_f}$, which is in $\Im(\tau_3)$  by point $i)$ in Lemma \ref{lema:trees-in-image}. Take the action by $G\in \Sp(H)$ that sends $a_{f'}$ to $a_{f'}+a_f$ and $b_f$ to $b_f-b_{f'}$. Then we have that
\[
G\Bigg(\tree{a_{f'},a_k,b_k,l,b_f}\Bigg)=\tree{a_{f'},a_k,b_k,l,b_f}-\tree{a_{f'},a_k,b_k,l,b_{f'}}
+\tree{a_f,a_k,b_k,l,b_f}-\tree{a_f,a_k,b_k,l,b_{f'}},
\]
and hence we get that
\[
\tree{a_f,a_k,b_k,l,b_f} -\tree{a_{f'},a_k,b_k,l,b_{f'}}
\]
belongs to $\Im(\tau_3)$.

\textbf{Type v)} Consider the tree $\tree{a_{f'},a_k,b_k,l,a_{f'}}$ which is in $\Im(\tau_3)$  by point $i)$ in Lemma \ref{lema:trees-in-image}.
Take the action by $G\in \Sp(H)$ that sends $a_{f'}$ to $a_{f'}+b_{f'}$. Then we have that the following element also belongs to $\Im(\tau_3)$.
\[
G\Bigg(\tree{a_{f'},a_k,b_k,l,a_{f'}}\Bigg)=
\tree{a_{f'},a_k,b_k,l,a_{f'}}
+\tree{a_{f'},a_k,b_k,l,b_{f'}}
+\tree{b_{f'},a_k,b_k,l,a_{f'}}
+\tree{b_{f'},a_k,b_k,l,b_{f'}}.
\]
Hence we get that 
\[
\tree{a_{f'},a_k,b_k,l,b_{f'}}
+\tree{b_{f'},a_k,b_k,l,a_{f'}} \in  \Im(\tau_3),
\]
and since we dealt with element of type $iv)$ we get by subtracting the right element that
\[
\tree{a_{f},a_k,b_k,l,b_{f}}
+\tree{b_{f'},a_k,b_k,l,a_{f'}} \in \Im(\tau_3).
\]

\textbf{Type vi)} By the IHX relation and point $iv)$, modulo $\Im(\tau_3)$, we have that
\[
\tree{a_f,a_k,b_k,l,b_f} - \tree{a_{f'},b_k,a_k,l,b_{f'}}=
\tree{a_f,a_k,b_k,l,b_f} - \tree{a_{f'},a_k,b_k,l,b_{f'}} -\tree{a_{k},b_k,a_{f'},l,b_{f'}}
\equiv \tree{l,b_{f'},a_{f'},b_k,a_k}
\]
which is in the $\mathfrak{N}_g$-orbit of $\tree{l,a_1,b_1,a_2,b_2}$ and hence in $\Im(\tau_3)$ by point $ii)$ in Lemma \ref{lema:trees-in-image}.

Finally, when there are two contractions, $h\cdot w=a_ka_kb_k$ or $a_kb_kb_k$. Then it is enough to prove that for any $f,f'$ different from $k$,  
\[
vii)\; \tree{a_f,a_k,b_k,a_k,b_f} - \tree{a_{f'},a_k,b_k,a_k,b_{f'}}
\qquad
viii)\; \tree{a_f,b_k,a_k,b_k,b_f} - \tree{a_{f'},b_k,a_k,b_k,b_{f'}} 
\]
belong to $\Im(\tau_3)$.

\textbf{Type vii)} Consider the tree $\tree{a_{f'},a_k,b_k,a_k,b_f}$, which belongs to $\Im(\tau_3)$ by point $iii)$ in Lemma \ref{lema:trees-in-image}. Take the action by $G\in \Sp(H)$ that sends $a_{f'}$ to $a_{f'}+a_f$ and $b_f$ to $b_f-b_{f'}$. Then we have that the following element belongs to $\Im(\tau_3)$.
\[
G\Bigg(\tree{a_{f'},a_k,b_k,a_k,b_f}\Bigg)=\tree{a_{f'},a_k,b_k,a_k,b_f}
-\tree{a_{f'},a_k,b_k,a_k,b_{f'}}
+\tree{a_{f},a_k,b_k,a_k,b_f}-\tree{a_{f},a_k,b_k,a_k,b_{f'}}.
\]
Hence we get that
\[
\tree{a_f,a_k,b_k,a_k,b_f} - \tree{a_{f'},a_k,b_k,a_k,b_{f'}}
\]
belongs to $\Im(\tau_3)$.

\textbf{Type viii)} The same argument of $vii)$ holds for this case exchanging $a_k$ and $b_k$.

\end{proof}
%%%%%%%%%%%%%%%%%%%%%%%%%%%%%%%%%%%%%%%%%%%%%%%%%%%%%%%%%%%%%%%%

%%%%%%%%%%%%%%%%%%%%%%%%%%%%%%%%%%%%%%%%%%%%%%%%%%%%%%%%%%%%%%%%
\begin{lemma}
\label{lema:0-contract}
%%%%%%%%%%%%%%%%%%%%%%%%%%%%%%%%%%%%%%%%%%%%%%%%%%%%%%%%%%%%%%%%
If $T$ has no contraction, $T-s(\Tr^M(T))$ belongs to $\Im(\tau_3).$
\end{lemma}
%%%%%%%%%%%%%%%%%%%%%%%%%%%%%%%%%%%%%%%%%%%%%%%%%%%%%%%%%%%%%%%%

%%%%%%%%%%%%%%%%%%%%%%%%%%%%%%%%%%%%%%%%%%%%%%%%%%%%%%%%%%%%%%%%
\begin{proof}
Let $T=\tree{a,b,c,d,e}$, then $T-s(\Tr^M (T))=T$ which is in $\Im(\tau_3)$ by Lemma \ref{lema:trees-without-contract}.

\end{proof}
%%%%%%%%%%%%%%%%%%%%%%%%%%%%%%%%%%%%%%%%%%%%%%%%%%%%%%%%%%%%%%%%

%%%%%%%%%%%%%%%%%%%%%%%%%%%%%%%%%%%%%%%%%%%%%%%%%%%%%%%%%%%%%%%%
\begin{lemma}
\label{lema:1-contract}
%%%%%%%%%%%%%%%%%%%%%%%%%%%%%%%%%%%%%%%%%%%%%%%%%%%%%%%%%%%%%%%%
If $T$ has a single contraction, $T-s(\Tr^M (T))$ belongs to $\Im(\tau_3).$
\end{lemma}
%%%%%%%%%%%%%%%%%%%%%%%%%%%%%%%%%%%%%%%%%%%%%%%%%%%%%%%%%%%%%%%%

%%%%%%%%%%%%%%%%%%%%%%%%%%%%%%%%%%%%%%%%%%%%%%%%%%%%%%%%%%%%%%%%
\begin{proof}
The trees $T$ with a single contraction are in the $\mathfrak{N}_g$-orbits of the trees:
\[
i)\; \tree{a_1,x,y,z,b_1}, \qquad ii)\; \tree{x,a_1,b_1,y,z}, \qquad iii)\; \tree{a_1,b_1,x,y,z}.
\]
with $x,y,z\in \mathcal{C}$ with no contractions such that $x\leq y \leq z$.

\textbf{Type i)} For trees of type $i)$ we have that
\[
T-s(\Tr^M (T))=\tree{a_1,x,y,z,b_1}-\tree{a_f,x,y,z,b_f},
\]
which belongs to $\Im(\tau_3)$ as shown in point $i)$ of the proof of Lemma \ref{lema:equivariance}.

\textbf{Type ii)} For the trees of type $ii)$ we have that
\[
T-s(\Tr^M (T))=T=\tree{x,a_1,b_1,y,z},
\] 
which belongs to $\Im(\tau_3)$ by point $i)$ in Lemma \ref{lema:trees-in-image}.

\textbf{Type iii)} For the trees of type $iii)$ we have that
\[
T-s(\Tr^M (T))=T=\tree{a_1,b_1,x,y,z},
\] 
which belongs to $\Im(\tau_3)$ by point $ii)$ in Lemma \ref{lema:trees-without-contract}.

\end{proof}
%%%%%%%%%%%%%%%%%%%%%%%%%%%%%%%%%%%%%%%%%%%%%%%%%%%%%%%%%%%%%%%%

%%%%%%%%%%%%%%%%%%%%%%%%%%%%%%%%%%%%%%%%%%%%%%%%%%%%%%%%%%%%%%%%
\begin{lemma}
\label{lema:2-contract}
%%%%%%%%%%%%%%%%%%%%%%%%%%%%%%%%%%%%%%%%%%%%%%%%%%%%%%%%%%%%%%%%
If $T$ has only 2 contractions, $T-s(\Tr^M(T))$ belongs to $\Im(\tau_3).$
\end{lemma}
%%%%%%%%%%%%%%%%%%%%%%%%%%%%%%%%%%%%%%%%%%%%%%%%%%%%%%%%%%%%%%%%

%%%%%%%%%%%%%%%%%%%%%%%%%%%%%%%%%%%%%%%%%%%%%%%%%%%%%%%%%%%%%%%%
\begin{proof}
The trees $T$ with only 2 contractions are in the $\mathfrak{N}_g$-orbits of the trees:
\[
i)\; \tree{a_1,b_1,b_1,x,y}, \qquad ii)\; \tree{a_1,b_1,x,b_1,y}, \qquad iii)\; \tree{x,a_1,b_1,b_1,y}, \qquad iv)\;\tree{x,b_1,a_1,b_1,y},
\]
\[
v)\; \tree{a_1,b_1,a_2,b_2,x}, \qquad vi)\; \tree{a_1,b_1,x,a_2,b_2}, \qquad vii)\; \tree{a_1,a_2,x,b_2,b_1} , \qquad viii)\; \tree{a_2,a_1,b_1,x,b_2}.
\]

For the first four types of trees $i)$-$iv)$ we show that, modulo $\Im(\tau_3)$, these trees can be written as a sum of trees with 1 contraction and hence by Lemma \ref{lema:1-contract} we will get that for these trees, $T-s(\Tr^M(T))$ belongs to $\Im(\tau_3)$.

\textbf{Type i)} For these trees we have that
\[
\tree{a_1,b_1,b_1,x,y}=\Bigg[ \frac{1}{2}\tree{a_1,b_1,b_1,a_1}, \tree{b_1,x,y}\Bigg] \in \Im(\tau_3).
\]

\textbf{Type ii)} For these trees we have that
\[
\tree{a_1,b_1,x,b_1,y}= \Bigg[  \tree{a_1,b_1,a_2}, \tree{b_2,x,b_1,y}\Bigg]-\tree{a_2,b_1,y,b_2,x}.
\]
Hence, modulo $\Im(\tau_3)$, we get a tree with just 1 contraction.

\textbf{Type iii)} Consider the tree $\tree{x,a_1,b_1,b_2,y}$ which belongs to $\Im(\tau_3)$ by Lemma \ref{lema:trees-in-image}. Take the action by
$G\in Sp(H)$ such that sends $b_2$ to $b_2+b_1$ and $a_1$ to $a_1-a_2$. Then the following element also belongs to $\Im(\tau_3)$.
\[
G\Bigg(\tree{x,a_1,b_1,b_2,y}\Bigg)=
\tree{x,a_1,b_1,b_2,y}
+\tree{x,a_1,b_1,b_1,y}
-\tree{x,a_2,b_1,b_2,y}
-\tree{x,a_2,b_1,b_1,y}.
\]
By Lemma \ref{lema:0-contract} the fourth tree belongs to $\Im(\tau_3)$ and hence, modulo $\Im(\tau_3)$, we get that
\[
\tree{x,a_1,b_1,b_1,y}\equiv
\tree{x,a_2,b_1,b_2,y}-\tree{x,a_1,b_1,b_2,y}.
\]

\textbf{Type iv)} For these trees we have that by the IHX relation,
\[
\tree{x,b_1,a_1,b_1,y} = \tree{x,a_1,b_1,b_1,y}+\tree{a_1,b_1,x,b_1,y}.
\]
and hence a sum of trees of types ii), iii) that we have already treated.

We now show that the trees of types v)-vii), modulo trees with a single contraction, are sums of trees of type viii).

\textbf{Type v)}  For these trees, by the AS and IHX relations,
\[
\tree{a_1,b_1,a_2,b_2,x}=\tree{a_1,a_2,b_1,b_2,x}+\tree{a_2,b_1,a_1,b_2,x}
=\tree{a_2,a_1,b_1,x,b_2}-\tree{a_2,b_1,a_1,x,b_2},
\]
which is a sum of trees of type $viii)$.

\textbf{Type vi)}  For these trees, by the AS and IHX relations,
\[
\tree{a_1,b_1,x,a_2,b_2}=\tree{a_1,b_1,a_2,x,b_2}+\tree{a_1,b_1,b_2,a_2,x}=
-\tree{a_1,b_1,a_2,b_2,x}+\tree{a_1,b_1,b_2,a_2,x},
\]
which is a sum of trees of type $v)$ and hence of type $viii)$.

\textbf{Type vii)}  For these trees, by the AS and IHX relations,
\[
\tree{a_1,a_2,x,b_2,b_1}=\tree{a_1,a_2,b_2,x,b_1}+\tree{a_1,a_2,b_1,b_2,x}=
\tree{a_1,a_2,b_2,x,b_1}+\tree{a_2,a_1,b_1,x,b_2},
\]
which is a sum of trees of type $viii)$.

Finally we show that for trees $T$ of type $viii)$, $T-s(\Tr^M(T))$ belongs to $\Im(\tau_3)$. As a consequence, by Lemma \ref{lema:1-contract} we get that for trees $T$ of types v)-viii), the element $T-s(Tr^M(T))$ also belongs to $\Im(\tau_3)$. 

\textbf{Type viii)}  For these trees we have that
\[
T-s(\Tr^M(T))= \tree{a_2,a_1,b_1,x,b_2}-\tree{a_f,a_1,b_1,x,b_f},
\]
which belongs to $\Im(\tau_3)$ as shown in point $iv)$ of the proof of Lemma \ref{lema:equivariance}.

\end{proof}
%%%%%%%%%%%%%%%%%%%%%%%%%%%%%%%%%%%%%%%%%%%%%%%%%%%%%%%%%%%%%%%%

%%%%%%%%%%%%%%%%%%%%%%%%%%%%%%%%%%%%%%%%%%%%%%%%%%%%%%%%%%%%%%%%
\begin{lemma}
If $T$ has only 3 contractions, $T-s(\Tr^M(T))$ belongs to $\Im(\tau_3).$
\end{lemma}
%%%%%%%%%%%%%%%%%%%%%%%%%%%%%%%%%%%%%%%%%%%%%%%%%%%%%%%%%%%%%%%%

%%%%%%%%%%%%%%%%%%%%%%%%%%%%%%%%%%%%%%%%%%%%%%%%%%%%%%%%%%%%%%%%
\begin{proof}
The trees $T$ with only 3 contractions are in the $\mathfrak{N}_g$-orbits of the trees:
\[
i)\; \tree{a_1,b_1,b_1,b_1,x}, \qquad ii)\; \tree{a_1,b_1,a_2,b_2,b_1}, \qquad iii)\; \tree{a_1,b_1,b_1,a_2,b_2}, \]
\[
iv)\;\tree{a_2,b_1,a_1,b_1,b_2},
\qquad v)\;\tree{a_2,a_1,b_1,b_1,b_2}.
\]

For the first three types of trees $i)$-$iii)$ we show that, modulo $\Im(\tau_3)$, these trees can be written as a sum of trees with 1 or 2 contractions and hence by Lemmas \ref{lema:1-contract} and \ref{lema:2-contract}  we will get that for these trees, $T-s(\Tr^M(T))$ belongs to $\Im(\tau_3)$.

\textbf{Type i)}
Consider the tree $\tree{a_1,b_1,b_1,b_2,x}$.
A direct computation shows that the Morita trace $\Tr^M$ vanishes on such tree, hence by Lemma \ref{lema:2-contract} such tree belongs to $\Im(\tau_3)$.
Take the action by
$G\in Sp(H)$ that sends $a_1$ to $a_1+a_2$ and $b_2$ to $b_2-b_1$.
Then we have that the following element also belongs to $\Im(\tau_3)$.
\[
G\Bigg(\tree{a_1,b_1,b_1,b_2,x}\Bigg)=
\tree{a_1,b_1,b_1,b_2,x}
-\tree{a_1,b_1,b_1,b_1,x}
+\tree{a_2,b_1,b_1,b_2,x}
-\tree{a_2,b_1,b_1,b_1,x}.
\]
Hence, by Lemmas \ref{lema:0-contract} and 
\ref{lema:2-contract}, we get that, modulo $\Im(\tau_3)$,
\[
\tree{a_1,b_1,b_1,b_1,x}
\equiv\tree{a_2,b_1,b_1,b_2,x}.
\]

\textbf{Type ii)}
Consider the tree
$\tree{a_1,b_1,a_2,b_2,b_3}$, which belongs to $\Im(\tau_3)$ by Lemma \ref{lema:trees-in-image}.
Take the action by $G\in Sp(H)$ that sends $b_3$ to $b_3+b_1$ and $a_1$ to $a_1-a_3$. Then we have that the following element also belongs to $\Im(\tau_3)$.
\[
G\Bigg(\tree{a_1,b_1,a_2,b_2,b_3}\Bigg)=
\tree{a_1,b_1,a_2,b_2,b_3}
+\tree{a_1,b_1,a_2,b_2,b_1}
-\tree{a_3,b_1,a_2,b_2,b_3}
-\tree{a_3,b_1,a_2,b_2,b_1}.
\]
Hence, by Lemmas \ref{lema:1-contract} and \ref{lema:2-contract}, we get that
modulo $\Im(\tau_3)$,
\[
\tree{a_1,b_1,a_2,b_2,b_1} \equiv \tree{a_3,b_1,a_2,b_2,b_3}.
\]

\textbf{Type iii)} For these trees we have that
\[
\tree{a_1,b_1,b_1,a_2,b_2}=
\Bigg[\frac{1}{2}\tree{a_1,b_1,b_1,a_1},\tree{b_1,a_2,b_2}\Bigg].
\]
Hence these trees belong to $\Im(\tau_3)$.

\textbf{Type iv)} For these trees we have that
\[
T-s(\Tr^M(T))= \tree{a_2,b_1,a_1,b_1,b_2}-\tree{a_f,b_1,a_1,b_1,b_f}.
\]
which belongs to $\Im(\tau_3)$ as shown in point $viii)$ of the proof of Lemma \ref{lema:equivariance}.

\textbf{Type v)} By the AS and IHX relations we have that
\[
\tree{a_2,a_1,b_1,b_1,b_2}
=\tree{a_2,b_1,a_1,b_1,b_2}+\tree{b_1,a_1,a_2,b_1,b_2}
=\tree{a_2,b_1,a_1,b_1,b_2}-\tree{b_1,a_1,a_2,b_2,b_1}.
\]
Therefore the trees of this type are the sum of trees of type ii) and iv).

\end{proof}
%%%%%%%%%%%%%%%%%%%%%%%%%%%%%%%%%%%%%%%%%%%%%%%%%%%%%%%%%%%%%%%%

%%%%%%%%%%%%%%%%%%%%%%%%%%%%%%%%%%%%%%%%%%%%%%%%%%%%%%%%%%%%%%%%
\begin{lemma}
If $T$ has only 4 contractions, $T-s(\Tr^M(T))$ belongs to $\Im(\tau_3).$
\end{lemma}
%%%%%%%%%%%%%%%%%%%%%%%%%%%%%%%%%%%%%%%%%%%%%%%%%%%%%%%%%%%%%%%%

%%%%%%%%%%%%%%%%%%%%%%%%%%%%%%%%%%%%%%%%%%%%%%%%%%%%%%%%%%%%%%%%
\begin{proof}
The trees $T$ with 4 contractions are the $\mathfrak{N}_g$-orbits of the trees:
\[
i)\; \tree{a_1,b_1,b_1,b_1,b_1}, \qquad ii)\; \tree{b_1,b_1,a_1,b_1,b_1}, \qquad iii)\;\tree{a_1,a_1,x,b_1,b_1},
\qquad iv)\;\tree{a_1,b_1,x,b_1,a_1}, \]
\[
v)\;\tree{a_1,a_1,b_1,b_1,x},
\qquad vi)\;\tree{a_1,b_1,a_1,b_1,x},
\]
with $x\in \mathcal{C}$ such that $\col(x)\neq 1$.

From all these trees, the first five represent the trivial derivation. Hence the last tree is the only one remaining:
\[
T := \tree{a_1,b_1,a_1,b_1,x}.
\] By direct computation $T -s(\Tr^M(T)) = \tree{a_1,b_1,a_1,b_1,x} +\tree{a_f,a_1,b_1,x,b_f}$ with $f \neq 1,\col(x).$ We prove that this element belongs to $\Im(\tau_3)$.
Consider the tree $\tree{a_f,b_1,a_1,b_1,x}$ which belongs to $\Im(\tau_3)$ as shown in point $viii)$ of the proof of Lemma \ref{lema:equivariance}. Take the action by $G\in Sp(H)$ that sends $a_f$ to $a_f + a_1$ and $b_1$ to $b_1-b_f$. Then we get that the derivation \begin{align*}
   G\Bigg(\tree{a_f,b_1,a_1,b_1,x}\Bigg)= &\tree{a_f,b_1,a_1,b_1,x} + \tree{a_f,b_1,a_1,-b_f,x} + \tree{a_f,-b_f,a_1,b_1,x} +\tree{a_f,b_f,a_1,b_f,x} \\
    & + \tree{a_1,b_1,a_1,b_1,x} + \tree{a_1,b_1,a_1,-b_f,x} + \tree{a_1,-b_f,a_1,b_1,x} +\tree{a_1,b_f,a_1,b_f,x}
\end{align*} belongs to $\Im(\tau_3)$.
A direct computation shows that the Morita trace vanishes on the sum of these trees minus the second and the fifth tree, which is formed by trees with $0$ or $2$ contractions. Then, by Lemmas \ref{lema:0-contract} and
\ref{lema:2-contract}, such a sum belongs to $\Im(\tau_3)$ and hence we deduce that 
\[
\tree{a_1,b_1,a_1,b_1,x} + \tree{a_f,b_1,a_1,x,b_f}\in \Im(\tau_3).
\]
Finally, by Lemma \ref{lema:2-contract}, the element $\tree{a_f,a_1,b_1,x,b_f}-\tree{a_f,b_1,a_1,x,b_f}$ belongs to $\Im(\tau_3)$, and hence adding this element to the element above we conclude the proof.
\end{proof}
%%%%%%%%%%%%%%%%%%%%%%%%%%%%%%%%%%%%%%%%%%%%%%%%%%%%%%%%%%%%%%%%

%%%%%%%%%%%%%%%%%%%%%%%%%%%%%%%%%%%%%%%%%%%%%%%%%%%%%%%%%%%%%%%%
\begin{lemma}
Any tree with at least $5$ contractions yields a trivial derivation.
\end{lemma}
%%%%%%%%%%%%%%%%%%%%%%%%%%%%%%%%%%%%%%%%%%%%%%%%%%%%%%%%%%%%%%%%

%%%%%%%%%%%%%%%%%%%%%%%%%%%%%%%%%%%%%%%%%%%%%%%%%%%%%%%%%%%%%%%%
\begin{proof}
Assume there is at least $5$ contractions, and one of the contractions is $a_k$ with $b_k$. If another one is $a_l$ with $b_l$ for $k \neq l$, then the last label should contract $3$ times with another label which is impossible.
Hence, the possible labels are:
\[
\lbrace a_k,a_k,b_k,b_k,b_k\rbrace, \quad 
\lbrace a_k,a_k,a_k,b_k,b_k\rbrace.
\]
Therefore the trees with at least 5 contractions are:
\[
\tree{a_k,a_k,b_k,b_k,b_k}, \quad
\tree{a_k,b_k,b_k,b_k,a_k}, \quad
\tree{b_k,b_k,a_k,a_k,a_k}, \quad
\tree{b_k,a_k,a_k,a_k,b_k}.
\]
Nevertheless all these trees represent the trivial derivation.
\end{proof}
%%%%%%%%%%%%%%%%%%%%%%%%%%%%%%%%%%%%%%%%%%%%%%%%%%%%%%%%%%%%%%%%

%Then the following will conclude.

%%%%%%%%%%%%%%%%%%%%%%%%%%%%%%%%%%%%%%%%%%%%%%%%%%%%%%%%%%%%%%%%
%\begin{lemma}
%If $\col(x) \neq k$, the difference $$\tree{a_k,b_k,b_k,a_k,c} + \tree{a_f,a_k,b_k,c,b_f} \text{ with } f := f(\{k,\operatorname{col(l)}\})$$ belongs to $\Im(\tau_3)$.
%\end{lemma}
%%%%%%%%%%%%%%%%%%%%%%%%%%%%%%%%%%%%%%%%%%%%%%%%%%%%%%%%%%%%%%%%

\appendix

\section{Tree-shaped Jacobi diagrams and symplectic derivations}
\label{S:AppendixA}
Here is a quick reminder about the diagrammatic interpretation of symplectic derivations. Once again we refer to \cite[Section 2.1]{FM22} for an overview, and to \cite{Lev06} and \cite{CST12} for the historical proofs. Recall that $\mathcal L(H)$ is the free Lie algebra generated by $H$ and that $D(H)$ is the Lie algebra of symplectic derivations. Using the identification of $H$ with its dual through the symplectic form $\omega$, \[ D_k(H) = \Ker\!\Big( H\otimes \mathcal L_{k+1}(H) \xrightarrow{[\ ,\ ]} \mathcal L_{k+2}(H) \Big). \] Following Levine \cite{Lev06}, degree-$k$ symplectic derivations admit a description in terms of tree-shaped Jacobi diagrams. Let $\mathcal T_k(H)$ be the free abelian group generated by connected unitrivalent trees with $k+2$ labelled leaves, modulo the AS, IHX and multilinearity relations. Given such a tree, choosing a leaf as a root turns the diagram into an iterated Lie bracket. Summing over all possible roots defines a homomorphism \[ \eta_k:\mathcal T_k(H) \longrightarrow H\otimes \mathcal L_{k+1}(H). \] Its image lies in $D_k(H)$ and, after tensoring with $\mathbb Q$, $\eta_k$ becomes an isomorphism. Integrally, however, the situation is more subtle. 
\begin{fact}[Degree one] A degree-$1$ tree has three leaves. Modulo the AS relation, permuting the labels changes the sign according to the signature of the permutation. It is known that \[ D_1(H)\simeq \Lambda^3 H. \] \end{fact} 

\begin{fact}[Degree two] The image of $\eta_2(\tree{a,b,a,b})$ is divisible by $2$ in $D_2(H)$. Thus the integral tree model for $D_2(H)$ contains, besides ordinary trees, the additional generators \[ \frac12\,\tree{a,b,a,b}. \] \end{fact} 

\begin{fact}
Let $k$ be odd. Suppose a tree is symmetric of the shape \[\tree{T,x,T}\] where the two copies of $T$ are identical. Then the corresponding derivation is trivial.
\end{fact}

Finally, since it is well-known that the Lie bracket of two derivations is a derivation, we need to translate this at the level of trees. If $T_1$ and $T_2$ are two trees, their bracket is defined by summing over all possible gluings of a leaf of $T_1$ with a leaf of $T_2$, weighted by the symplectic pairing of the corresponding labels. More precisely, \[ [T_1,T_2] = \sum_{v_1,v_2} \omega(\ell(v_1),\ell(v_2)) \,T_1\ast_{v_1,v_2} T_2, \] where the sum runs over all leaves $v_1$ of $T_1$ and $v_2$ of $T_2$, $\ell(v_i)$ denotes the label of $v_i$, and $T_1\ast_{v_1,v_2} T_2$ is the tree obtained by gluing the leaves $v_1$ and $v_2$.

\begin{fact}
Under this bracket, the map \[ \eta:=\bigoplus_{k\geq 1}\eta_k \] is a homomorphism of graded Lie algebras. Equivalently, if $T_1$ and $T_2$ are two tree diagrams, then \[ \eta([T_1,T_2]) = [\eta(T_1),\eta(T_2)], \] where the bracket on the right-hand side is the commutator of derivations. Consequently, computations in $D(H)$ may be performed diagrammatically.
\end{fact}

\bibliographystyle{alpha}
\bibliography{biblio}
\end{document}